\documentclass[11pt,reqno]{amsart}
\usepackage{amsmath}
\usepackage{amssymb}
\usepackage{mhequ}
\usepackage{color}
\usepackage{graphicx}
\usepackage{enumerate}
\usepackage{todonotes}
\usepackage[unicode=true,
 bookmarks=true,bookmarksnumbered=false,bookmarksopen=false,
 breaklinks=false,pdfborder={0 0 1},backref=false,colorlinks=true]{hyperref}

\newtheorem{maintheorem}{Theorem}

\newtheorem{theorem}{Theorem}[section] % 1st argument is your name for it
\newtheorem{lemma}[theorem]{Lemma}     % 2nd argument is what is printed
\newtheorem{corollary}[theorem]{Corollary}
\newtheorem{proposition}[theorem]{Proposition}
\newtheorem{remark}{Remark}

\newtheorem{example}{Example}
\newcommand{\Vol}{\operatorname{Vol}}
\newcommand{\Id}{\operatorname{Id}}
\newcommand{\esssup}{\operatorname{ess\,sup}}
\newcommand{\set}[1]{\{#1\}}

\begin{document}

\title[Representation of measures]% end with percent
 {On the regular representation of measures} 

\author[J. Jost]{J\"{u}rgen Jost}
\address{J\"{u}rgen Jost\\
   Max-Planck-Institute for Mathematics in the Sciences\\
   Inselstr. 22\\
   04103 Leipzig\\
   Germany}
   \email{jost@mis.mpg.de}

\author[R. Matveev]{Rostislav Matveev}
\address{Rostislav Matveev\\
   Max-Planck-Institute for Mathematics in the Sciences\\
   Inselstr. 22\\
   04103 Leipzig\\
   Germany}
   \email{matveev@mis.mpg.de}
   
\author[J.W. Portegies]{Jacobus W.~Portegies}
\address{Jacobus W.~Portegies\\
   Max-Planck-Institute for Mathematics in the Sciences\\
   Inselstr. 22\\
   04103 Leipzig\\
   Germany}
   \email{jacobus.portegies@mis.mpg.de}

\author[C.S. Rodrigues]{Christian S.~Rodrigues}
\address{Christian S.~Rodrigues\\
Institute of Mathematics\\
Universidade Estadual de Campinas\\
   13.083-859 Campinas - SP\\
   Brazil}
   \email{rodrigues@ime.unicamp.br}

\date{\today}

\begin{abstract}
We give sufficient conditions for a parametrised family of probability measures on a Riemannian manifold with boundary to be represented by random maps of class $C^k$. 
The conditions allow for the probability densities to approach zero towards the boundary of the manifold. 
We also formulate two obstructions to regular representability.
\end{abstract}

\keywords{Markov chain, random dynamics, random maps, representation of measures}
\subjclass{Primary:  37C40, 49K45, 49N60, Secondary: 37H10, 37C05}

\maketitle

\section{Introduction}

A representation of a family of probability measures $\{\mu_x\}_{x \in X}$ on a Riemannian manifold $M$, parametrised by a Riemannian manifold $X$, is a mapping $T:
X \times \Omega \to M$, such that for each
$x \in X$, it holds that
\begin{equation}
\label{eq.repres}
\mu_{x} = T(x,\cdot)_{*}\mathbb{P},
\end{equation}
where $(\Omega, \mathbb{P})$ is an auxiliary probability space, and
$T(x,\cdot)_{*}\mathbb{P}$ denotes the push forward of $\mathbb{P}$ by
$T(x,\cdot)$. 

In this paper we will investigate under which assumptions we can guarantee the existence of a representation such that the maps $T(\cdot, \omega)$ are $C^k$-regular, uniformly in $\omega$.
In this case, we will say that the family $\{\mu_x\}_{x \in X}$ is boundedly $C^k$-representable.

This question appears in several contexts in Ergodic Theory~\cite{Kif86,Kif88,RoV16}. In particular, it is relevant in the study of stochastic stability of dynamical systems. 
Roughly speaking, stochastic stability refers to the stability of asymptotic states of a dynamical system under small random perturbations.

Given a discrete-time dynamics generated by a map $f: M \to M$, a natural way to introduce random perturbations is by the Markov chain model, that is to replace the map $f$ by a Markov kernel $m \mapsto \mu_m^\varepsilon$, in that the image of a point $m$ under iteration of the system is distributed according to the distribution $\mu_m^\varepsilon$. 
The distribution $\mu_m^\varepsilon$ is usually assumed to be localised near $f(m)$, and $\varepsilon$ is thought of as the scale of the localisation.

To show that a given dynamics is stochastically stable, it is for technical reasons often more convenient to introduce random perturbations by a model of random maps  \cite{Kif86,Kif88,You86,Ara00,AlA03,BeV06,AAV07}. 
In the random maps model, in each iteration of the dynamical system, the map is chosen at random according to a certain distribution. 

Arguably, the Markov chain model is more natural than the random maps model. 
This raises the question under which conditions a Markov chain model can be realised by a model of random maps. 
In addition, regularity properties of the random maps are essential for proving statements about stochastic stability. 
Therefore, it is important to investigate the relationship between regularity of the Markov kernels and the regularity of the random maps.

In fact, a realisation of a Markov chain model by a model of uniformly $C^k$ random maps, exactly corresponds to a bounded $C^k$  representation of the Markov kernels $\mu_m^\varepsilon$ (in the special case where the parametrising manifold $X$ is $M$ itself). To see this, note that when the family is boundedly $C^k$-representable, the map $L: \Omega \to C^k(X, M)$ given by
\[
L(\omega) (\cdot) := T(\cdot, \omega),
\]
induces a probability distribution $L_{*} \mathbb{P}$ on $C^k(X, M)$. 
If a map $f$ is drawn at random from $L_{*} \mathbb{P}$, the value of $f(x)$ will be distributed according to $\mu_x^\varepsilon$.

\subsection{Main result}
As the main result of 
this paper, we show that a family $\{\mu_{x}\}_{x\in X}$ is boundedly $C^k$-representable ($k \in \mathbb{N}$) in case $M$ is a compact, oriented,  and connected Riemannian manifold of class
$C^{k+2}$ with or without boundary $\partial M$, $X$ is a compact Riemannian manifold of class $C^k$, and the measures $\{\mu_x\}_{x \in
  X}$ are absolutely constinuous with respect to the volume measure on $X \times M$, with densities $\rho(\cdot, \cdot)$ of class $C^k$ that are positive on the interior of $M$ and have suitable decay towards the boundary $\partial M$.
The assumptions on the decay will be phrased in Section \ref{suse:assumptions}, and the result will be formally stated as Theorem \ref{th:CkRepresentability}. 

With respect to earlier work, which we will briefly review in the next section, the additional difficulties here are that the densities of the measures $\mu_x$ are not assumed to be uniformly bounded away from zero, and that there are neither convexity assumptions on the support of the measures nor curvature assumptions on the manifold $M$. 

We illustrate by several examples presented in Section \ref{se:Examples} though that some assumptions on the decay towards the boundary are then necessary. We phrase a more general obstruction to bounded Lipschitz representability in Theorem \ref{th:obstructionRepr}, and an obstruction to bounded $C^k$ representability in Proposition \ref{pr:ObstructionTesting}.

The main result can also be used as an ingredient to prove representability in situations where the supports of the measures $\mu_x$ are varying and are not necessarily equal to the full manifold $M$. 
We formulate a simple approach to representation in this case in Corollary~\ref{co:MovingMeasures}.

\subsection{Previous work on representation of measures}
Blumenthal and Corson seem to be the first to address the problem of representating a family of Markov kernels by random maps~\cite{BlC70, BlC72}.  
They showed that a weak-* continuous family of probability measures can be represented by \emph{continuous} random maps.
After that, Kifer~\cite[Theorem 1.1]{Kif86} gave conditions on a family of probability measures that guarantee the existence of a \emph{measurable} representation. Later,
Quas~\cite{Qua91} implicitly used the so-called Moser coupling \cite{Mos65} 
to show that a family of probability measures on a smooth, closed, orientable Riemannian manifold, with smooth densities that smoothly depend on the parameter,  can be represented by
\emph{smooth} random maps. 
Ara\'{u}jo~\cite[Examples 1, and
2]{Ara00} and Benedicks and Viana~\cite{BeV06} and~\cite[D.4]{BDV05},
also provided some examples of representation for restricted classes
of maps and probability families; see also~\cite{JKR15} for a more elaborate historical overview.

Kell and the first and last author studied the 
representation of Markov chains by random maps focusing on how the regularity of the maps depend on the regularity of the probability measures~\cite{JKR15}. 
They showed using techniques from optimal transport that $C^k$ transition kernels can be represented by $C^k$ random maps, under assumptions about the positivity of the probability measures on the support, and on the convexity of the support. 
Without these assumptions, maps produced by optimal transport do not seem to have enough regularity, and therefore in this paper, we follow a different approach.

Most of the above techniques for showing regular representation work according to the following scheme. 
Find a regular family of transport maps $T(x, \cdot):M \to M$ that push a reference measure $\mu$ to the measure $\mu_x$, and show regular dependence of the map $T(x, \cdot)$ on $x$. 
The map $T:X\times M \to M$ is then a representation, for which $T(\cdot, m)$ is regular.

This scheme actually yields more regularity than strictly necessary: for $C^k$-representability the requirement is that for fixed $m \in M$, the maps $T(\cdot, m)$ are regular, and the regularity in $m$ is not important. 
We leverage this observation to treat the transport in the interior of the manifold, where the densities are positive, by standard Moser coupling, and to transport the mass near the boundary of the manifold by a radial, monotone rearrangement that we can analyse carefully.

\section{Main representability results}
\label{sec.not}

Let $M$ be a compact, oriented, connected manifold of
class $C^{k+2}$ with or without boundary $\partial M$.  Let $\bar{\mu}
= \{\mu_x\}_{x \in X}$ be a family of probability measures,
parametrised by a manifold $X$ of class $C^k$.
The representation of the family $\bar{\mu}$ can be thought of as a
measurable map
\begin{equation}\label{e:F}
T:X\times \Omega\to M,
\end{equation}
such that for every $x \in X$,
\[
T(x, \cdot )_{*} \mathbb{P}=\mu_{x},
\]
where $\Omega=(\Omega,\mathbb{P})$ is a Rochlin-Lebesgue
probability space and there exists a constant $C > 0$, such that, for
every $\omega \in \Omega$, the $C^k$-seminorm of the map $x \mapsto T(x,
\omega)$ is bounded by $C$.

We may choose a Riemannian metric $g$ on $M$, such that its
coefficients written out in coordinates are of class $C^{k+1}$. We
denote by $\rho(x, \cdot)$ the density of $\mu_x$ with respect to the standard
Riemannian volume measure $\Vol$.  We will assume that $\rho \in C^k(X
\times M)$, and that $\rho > 0$ on the interior $\mathring{M}$ of $M$.

\begin{remark}\label{rmk}
  We note that by the compactness of $X$ and $M$, there exists a
  reference density $f \in C^{k+2}(M)$, $f > 0$ on the interior of $M$
  such that for all $x \in X$ and $m \in M$,
\[
\rho(x,m) > f(m).
\]
\end{remark}

\subsection{Description of the representation}

The probability space $\Omega$ will be concrete: it will be
either $M$ itself or a subset of it. We will also explicitly construct
$T$ as a map from $X \times M$ to $M$.  We distinguish two cases,
based on whether there is a positive uniform lower bound on the
densities or not.

In the special case where such a positive uniform lower bound exists,
we will construct a representation by showing regular dependence on
parameters in the standard Moser coupling.  This will also be an
ingredient in the proof of the more general case.  The precise
statement is formulated in the following proposition.

\begin{proposition}\label{p:moser1}
  Suppose $M$ is an oriented, compact, and connected Riemannian
  manifold (with or without boundary) of class $C^{k+2}$, with $k\geq
  1$, and the family $\bar{\mu}=\set{\mu_{x}}_{x\in X}$ is such that
  \begin{itemize}
  \item
    $\rho(\cdot,\cdot)\in C^{k}(X\times M)$
  \item
    $\rho(x,m)\geq c$ for some $c>0$ and every $(x,m)\in X\times M$.
  \end{itemize}
  Let $\mu = \rho_0 \text{d} \Vol_M$ with $\rho_0 \in C^k(M)$ with
  $\rho_0 \geq c$.  Then, there exists a map $F \in C^k(X\times M, M)$,
  such that, for every $x \in X$, the map $F(x, \cdot): M \to M$ is a diffeomorphism and
\[
F(x, \cdot)_{*} \mu = \mu_{x}.
\]
In particular, the family $\bar{\mu}$ is boundedly
$C^{k}$-representable.
\end{proposition}
The proof of this proposition is presented in
Section~\ref{se:withmaintheorem}.

In case there is no uniform lower bound for the densities on
$M$, we need to make some assumptions on how the
densities decay towards the boundary.

%%%%%%%%%%%%%%%%%%%%%%%%%%%%%%%%%%%%%%%%%%%%%%%%%%%%%%%%%%%%%%
\subsection{Assumptions on the decay of the densities}
\label{suse:assumptions}

Note that we may choose collar coordinates in a neighbourhood $U$ of
the boundary of $M$.
That is, there exists a $C^{k+2}$ parametrisation $Q : \partial M \times
[0,1] \to M$, which is a diffeomorphism onto its image, and restricts
to the identity on $\partial M\times \{0\}$.  For a point $(a,t)
\in \partial M \times [0,1]$, we will occasionally refer to $a$ as
the boundary coordinate, and to $t$ as the collar coordinate.
Moreover, we will at times abuse notation by for instance writing
$G(a,t)$ instead of $G(Q(a,t))$ for a function $G: M \to \mathbb{R}$.

As noted in Remark~\ref{rmk}, the compactness of $X$ and $M$ implies
the existence of a reference density $f \in C^{k+2}(M)$. In fact, by compactness, $f$
can be chosen in such a way that in the collar, $f$ is independent of
the boundary coordinate.  We define the reference measure $\mu := f
\text{d}\Vol_M$.

We assume that $\rho \in C^k(X \times M)$ and $\rho > 0$ on the
interior $\mathring{M}$ of $M$.  In order for any representability
result to hold, additional assumptions on the decay of the densities
are necessary.

We assume that there exist functions $E:\partial M \times [0,1]\to
[1,\infty)$ and $B:\partial M \times [0,1]\to [1,\infty)$ such that
for every multi-index $\beta$ and every index $j \in \mathbb{N}_0$
with $|\beta| + j \leq k$, and every $a \in \partial M$, it holds that
\begin{equation}%\tag{A1}
\label{eq:DerivativeAssumption}
\frac{1}{\rho(x,a,t)} \left| D_x^{\beta} D_t^{j}  \rho(x,a,t) \right| \leq E(a,t)^{|\beta|} B(a,t)^{j}
\end{equation}
and
\begin{equation}%\tag{A2}
\label{eq:IntegratedAssumption}
\frac{1}{\rho(x,a,t)} \int_0^t\left|D_x^{\beta} \rho(x,a,s) \right| \text{d} s \leq \frac{E(a,t)^{|\beta|}}{B(a,t)},
\end{equation}
and 
\begin{equation}%\tag{A3}
\label{eq:ClosureAssumption}
E(a,t)^k \leq A B(a,t)
\end{equation}
for some constant $A > 0$.

We will illustrate these
assumptions by some examples  in Section \ref{se:Examples}.

\subsection{Main Theorem}
\label{suse:maintheorem}
The following theorem states our main representability result.

\begin{maintheorem}\label{th:CkRepresentability}
  Let $M$ be a compact, oriented, connected manifold
  of class $C^{k+2}$, for some $k\in \mathbb{N}$, with boundary $\partial M$.  Let
  $\{\mu_x\}_{x \in X}$ be a family of probability measures
  parametrised by a manifold $X$ of class $C^k$, and suppose the
  assumptions in Section~\ref{suse:assumptions} hold.  Then, there
  exists a family of piecewise $C^k$-maps $T_x: M \to M$, such that
  $(T_x )_{*} \mu = \mu_x$, and the $C^k$-norm of $x \mapsto T_x(m)$
  is uniformly bounded in $m \in M$.  In particular, the family
  $\{\mu_x\}_{x\in X}$ is boundedly $C^k$-representable.
\end{maintheorem}

%%%%%%%%%%%%%%%%%%%%%%%%%%%%%%%%%%%%%%%%%%%%%%%%%%%%%%%%%%%%%%%
\subsection*{On the proof of the Main Theorem}

Whereas the representation $F$ in Proposition \ref{p:moser1} is
actually regular from $X \times M$ to $M$, when there is no lower
bound on the densities it will be important to exploit that bounded
representability only requires the regularity with respect to $X$, but
not with respect to $M$.

We will now describe the construction of the representing map $T$ for
the general case. If we would try to repeat the same argument
involving the Moser coupling as in the proof of Proposition
\ref{p:moser1}, the blow-up of derivatives with respect to the point
in $M$ also prohibits the control of the derivatives with respect to
the parameter $x$.

Therefore, we introduce a pre-processing step in the form of
Proposition \ref{pr:DiffCollar}.  Due to its length, we chose to only
list the proposition later, but the important conclusion is that we
may construct homeomorphisms $G_x$ from $M$ to itself, which are
diffeomorphisms on the interior of $M$, and which uniformise the behavior
of the measures $\mu_x$ near the boundary. More precisely,
$(G_x^{-1})_{*} \mu_x$ equals $\mu$ when restricted to a certain
neighbourhood $V$ of the boundary of $M$.  Moreover, the $C^k$-norm of
the maps $x \mapsto G_x(m)$ are bounded, uniformly in $m$.
 
We will see that in the complement of this neighbourhood $V$, even
though the measures $(G_x^{-1})_{*} \mu_x$ differ from $\mu$, their
densities are regular and uniformly bounded away from zero.  Hence, we
have reduced the problem to the special case and we can use Proposition
\ref{p:moser1} to obtain a representation $F$ of the family
$(G_x^{-1})_{*} \mu_x$ restricted to the complement of $V$.  We extend
$F$ to the identity on $V$.  Note that $F$ is only piecewise regular.
Now we define for every $x \in X$, the map $T_x$ by $T_x = G_x \circ
F_x$ and check that indeed
\[
(T_x )_{*} \mu = (G_x \circ F_x)_{*} \mu = (G_x)_{*} (F_x)_{*} \mu =
(G_x)_{*} (G_x^{-1})_{*} \mu_x = \mu_x.
\]

%%%%%%%%%%%%%%%%%%%%%%%%%%%%%%%%%%%%%%%%%%%%%%%%%%%%%%%%%%%%%%%

Below, we formulate an easy corollary that shows bounded
representability when the support $M_x$ of the measures $\mu_x$ is not
the full manifold, but rather an $n$-dimensional submanifold for which the boundary is of class $C^{k}$.  We assume that
$\mu_x$ is absolutely continuous with respect to the Riemannian volume
measure on $M$, and denote the density by $\rho(x, \cdot)$.  We assume
$\rho$ is positive on the interior of its support $\mathring{M}_x$. 
We also assume the existence of a compact, oriented, connected, Riemannian manifold $N$ of class $C^{k+2}$ with boundary of class $C^{k+1}$ such that $N$ is diffeomorphic to $M_x$ for all $x \in X$.

\begin{corollary}\label{co:MovingMeasures}
Let $N$ be a compact, oriented, connected Riemannian manifold of class $C^{k+2}$ with boundary of class $C^{k+1}$. 
Suppose there exists a family of $C^{k}$-diffeomorphisms $R_x: N \to
  M_{x}$, such that the map $R$ is of class $C^k(X
  \times N, M)$.  Let the family $(R_x^{-1})_{*} \mu_x$ satisfy
  the hypotheses of Theorem \ref{th:CkRepresentability}.  Then the family $\{\mu_x\}_{x \in X}$ is boundedly $C^k$-representable.
\end{corollary}

\begin{proof}
  By Theorem \ref{th:CkRepresentability} we obtain maps $T_x:N \to N$ such that 
\[
(T_x)_{*}\mu = (R_x^{-1})_{*} \mu_x
\]
and such that $x \mapsto T_x(m)$ is of class
  $C^k(X , M)$, with norm bounded uniformly in $m$.  Then $(R_x\circ T_x)_{*} \mu = \mu_x$, and $x \mapsto
  R_x(T_x(m))$ is of class $C^k(X,M)$ as well.

\end{proof}

%%%%%%%%%%%%%%%%%%%%%%%%%%%%%%%%%%%%%%%%%%%%%%%%%%%%%%%%%%%%%%%
\section{Moser coupling for measures with densities bounded away from
  zero}
\label{se:withmaintheorem}

In this section we will prove the $C^k$-representability of a family
of measures on a Riemannian manifold with boundary, under the
assumption that the densities of the measures with respect to the
Riemannian volume are uniformly bounded away from zero.\\
\textbf{Proposition~\ref{p:moser1}}\textit{   Suppose $M$ is an oriented, compact, and connected Riemannian
  manifold (with or without boundary) of class $C^{k+2}$, with $k\geq
  1$, and the family $\bar{\mu}=\set{\mu_{x}}_{x\in X}$ is such that
  \begin{itemize}
  \item
    $\rho(\cdot,\cdot)\in C^{k}(X\times M)$
  \item
    $\rho(x,m)\geq c$ for some $c>0$ and every $(x,m)\in X\times M$.
  \end{itemize}
  Let $\mu = \rho_0 \text{d} \Vol_M$ with $\rho_0 \in C^k(M)$ with
  $\rho_0 \geq c$.  Then, there exists a map $F \in C^k(X\times M, M)$,
  such that, for every $x \in X$, the map $F(x, \cdot): M \to M$ is a diffeomorphism and
\[
F(x, \cdot)_{*} \mu = \mu_{x}.
\]
In particular, the family $\bar{\mu}$ is boundedly
$C^{k}$-representable.
}

\begin{proof}[Proof of Proposition \ref{p:moser1}] The proof uses a
  parametric Moser coupling, together with the particular choice of
  the primitive form and analysis of its dependence on the
  parameter. Given two positive top-dimensional forms, the Moser
  coupling is used to find a self-diffeomorphism of the manifold that
  pulls back one form into another.  The diffeomorphism is found as a
  time-one-map of the flow, that ``follows'' the linear deformation of
  the forms.  Villani describes the Moser coupling in the language of
  partial differential equations rather than differential
  forms in his book \cite{Vil09}.

  Below we describe the skeleton of the construction and then retrace
  all the steps to take care of all the necessary bounds.

Note that by compactness it suffices to show the proposition for a conveniently chosen Riemannian metric $g$ on $M$. 
We choose a metric $g$ such that its coefficients in coordinates are of class $C^{k+1}$, and such that in collar coordinates $(a,s) \in \partial M \times [0,1]$, it holds that
\[
g(a,s) = g_{\partial M}(a) + \mathrm{d} s^2. 
\]
The latter is not essential but simplifies some steps at the end of this proof. 

We choose an orientation on $M$ and set
\[
\text{d} m:=\frac{\text{d}\Vol}{\Vol(M)}
\]
For the purposes of this proof we interpret the measures as positive
top-dimensional volume forms
on $M$:
\[
\begin{split}
\mu_{x}&=\rho_{x}\text{d} \Vol,\\
\mu &= \rho_0 \text{d} \Vol.
\end{split}
\]

\textbf{Step 1:} Consider a linear deformation connecting $\text{d} m$
to $\mu_{x}$. For $t\in[0,1]$ we define a $X$-family of $n$-forms
$\boldsymbol{\alpha}$ and an $I\times X$-family $\boldsymbol{\eta}$ by
\begin{align}
  \alpha_{x}
  &:=
 -  \mu + \mu_x
 % \Label{e:Ma}
  \\
  \eta_{t,x}
  &:=
  \mu + t\cdot\alpha_{x}
  %\Label{e:Me}
\end{align}

\textbf{Step 2:} Let $\boldsymbol{n}$ be a unit, inward normal to the
boundary vector-field defined along $\partial M$.  Define an
$X$-family $\boldsymbol{\gamma}$ of $n$-forms, letting $\gamma_x$
be the unique solution to
\begin{equation}
\label{eq:ProblemGammax}
\left\{
\begin{aligned}
\Delta \gamma_x = (\text{d} \delta + \delta \text{d} \,) \gamma_x &= \alpha_x \\
\iota_{\boldsymbol{n}} ( \star \delta \gamma_x) &= 0\\
\int_M \gamma_x &= 0,
\end{aligned}
\right.
\end{equation}
where $\Delta$ denotes the Hodge Laplacian, $\star$ the Hodge star operator, $\delta$ the
codifferential, and $\iota_X \omega$ the contraction of a form $\omega$ with a vector field $X$. Note that the existence of a solution is guaranteed
since $\int_M \alpha_x = 0$ (because all measures are probability measures) and $M$ is connected.

Define an $X$-family $\boldsymbol{\beta}$ of primitives by setting $\beta_x = -
\delta \gamma_x$.  
Since $\gamma_x$ is top-dimensional, it is automatically closed, that is $\mathrm{d} \gamma_x = 0$. Hence $\beta_x$ solves the system
\begin{equation}\label{e:Mb}
\left\{
  \begin{aligned}
    \text{d}\beta_{x}&= - \alpha_{x}\\
    \iota_{\boldsymbol{n}}(\star\beta_{x})&=0
  \end{aligned}
\right.
\end{equation}

\textbf{Step 3:} Define an $I\times X$-family $V$ of boundary-parallel
vector fields on $M$ as the solution to
\begin{equation}\label{e:Mv1}
 \iota_{ V_{t,x} }\eta_{t,x}=\beta_{x}
\end{equation}

\textbf{Step 4:} Solve the ODE
\begin{equation}\label{e:Mfi1}
\left\{
  \begin{aligned}
    \dot\Phi_{t,x}&=V_{t,x}\circ\Phi_{t,x}\\
    \Phi_{0,x}&=\Id_{M}
  \end{aligned}
\right.
\end{equation}
If we now push forward $\text{d} m$ by the flow, we end up with
$\eta_{t,x}$, that is
\[
(\Phi_{t,x})_{*} \mu = \eta_{t,x}.
\]
To check this, we calculate the Lie derivative of $\eta_{t,x}$ by the
Cartan formula
\[
\mathcal{L}_{V_{t,x}} \eta_{t,x} = \text{d}( \iota_{V_{t,x}}
\eta_{t,x} ) = \text{d} \beta_x = - \alpha_x = - \partial_t
\eta_{t,x}.
\]
Since by definition of the Lie derivative,
\[
\partial_t \left[(\Phi_{t,x})_{*} \mu \right] = \partial_t
\left[(\Phi_{t,x}^{-1})^{*} \mu \right] = - \mathcal{L}_{V_{t,x}}
[(\Phi_{t,x})_{*} \mu ],
\]
the forms $\eta_{t,x}$ and $(\Phi_{t,x})_{*} \mu$ satisfy the same PDE
with the same initial and boundary conditions, and solutions to this
PDE are unique.

Thus, the time-one-map defined by $F(x,m):=\Phi_{1,x}(m)$ is the
required parametrisation 
\begin{equation}\label{e:Mf1}
  F:X\times M\to M
\end{equation}
where $M$ in the domain of definition serves the role of $\Omega$ in
(\ref{e:F}).

We will now check that indeed, $F \in C^k( X \times M, M )$.

In order to treat the parameter manifold $X$ and the manifold $M$ more
symmetrically, we interpret the vector field $V$ as a time-independent
vector field on $I \times X \times M$, such that the projection to
$T_{(t,x,m)} M$ of $V$ coincides with $V_{t,x}$, the projection to
$T_{(t,x,m)} X$ vanishes, and the projection to $T_{(t,x,m)}I$ equals
the standard unit vector in $I$ pointing in the positive direction.

Below we will explain why $V \in C^k(I \times X \times M)$.  
Since $k \geq 1$, this
regularity for $V$ immediately implies our claim that $F \in C^k(X \times
M, M)$, by a standard regular
dependence on parameters argument (cf. \cite[Section 1.7]{CoL55}),

Note that problem (\ref{eq:ProblemGammax}) translates into the
following problem for the density $u_x$ of $\gamma_x$ with respect to
the Riemannian volume form (see~\cite[Appendix to Ch. 1]{Vil09})
\begin{equation}
\label{eq:ProblemInTermsOfu}
\left\{
\begin{aligned}
- \Delta u_x = - \star\text{d} \star \text{d} u_x &= \star \alpha_x = \rho_x - \rho_0 \\
\frac{\partial u_x}{\partial \boldsymbol{n}} &= 0\\
\int_M u_x \text{d} \Vol &= 0.
\end{aligned}
\right.
\end{equation}

Note that since $\beta_{t,x} = - \delta \gamma_x$, and $\iota_{V_{t,x}} \eta_{t,x} = \beta_{t,x}$,
\[
V_{t,x} = \frac{1}{\rho_0 + t (\rho_x - \rho_0) } \mathrm{grad} \, u_{x}.
\]

This is the first point at which the choice of metric $g$ becomes convenient. 
It allows us to easily localise to a (boundary) chart by using a cutoff function of class $C^{k+2}$ with vanishing normal derivative. 
We may therefore assume without loss of generality that $M$ is just a domain in the halfspace $\mathbb{R}^n_+$. 

Written out in coordinates, in non-divergence form, 
the coefficients
of the Laplace-Beltrami operator are of class $C^{1}(M)$.  Therefore,
for sufficiently regular solutions $u$ to
\begin{equation}
\label{eq:ProblemForu}
\left\{
\begin{aligned}
- \Delta u &= \rho,\\
\frac{\partial u}{\partial \boldsymbol{n}} &= 0,\\
\int_M u \text{d} \Vol &= 0.
\end{aligned}\right.
\end{equation}
the following a-priori estimate holds (cf. \cite[Ch. 9]{Fri08})
\[
\| u \|_{W^{2,p} (M)} \leq C \| \rho \|_{L^p(M)}.
\]
That the estimate also holds for the unique solution $u$ of (\ref{eq:ProblemForu}) follows for instance by an approximation argument.

Morrey's inequality states that for $p > n$ and $\sigma = 1 - p/n$,
\[
\| u \|_{C^{1,\sigma}(M)} \leq C \| \rho \|_{L^p(M)}.
\]
Hence for $x, y \in X$,
\[
\| V_{t,x} - V_{t,y} \|_{C^\sigma(M)} \leq C \| \rho_x - \rho_y  \|_{L^p(M)}.
\]
Since $\rho \in C^0(X \times M)$, the right-hand side can be made arbitrarily small by choosing $x$ close enough to $y$. 
From here, it follows that $V \in C^0(I \times X \times M )$.

To obtain higher, and joint regularity for $u_x$, one may proceed by induction on the order of the derivative. One differentiates problem (\ref{eq:ProblemInTermsOfu}) with respect to coordinates on $X$ and $M$, to get equations for the derivatives of
$u$. Again, the choice of the metric $g$ is convenient, as it implies that derivatives of $u$ also satisfy homogeneous boundary conditions. 

For instance, for multi-indices $k_1$ and $k_2$,  $|k_1| + |k_2| \leq k$, the equation for the derivative $D_x^{k_1} D_m^{k_2} u_x$ is schematically 
\[
\left\{
\begin{aligned}
- \Delta (D_x^{k_1} D_m^{k_2} u_x) & = \mathrm{r.h.s.}\\
\frac{\partial (D_x^{k_1} D_m^{k_2} u_x)}{ \partial \boldsymbol{n} }&= 0.
\end{aligned}
\right.
\]
The right-hand side r.h.s.~consists of a term with derivatives of $\rho$ and terms that are derivatives of the metric coefficients multiplied by derivatives of $u$. 
Using that $\rho \in C^k( X \times M)$, the metric coefficients are of class $C^{k+1}(M)$, and lower-order derivatives of $u$ are controlled by the induction hypotheses, we conclude that $\mathrm{r.h.s.} \in L^p(X \times M)$ for every $p > n$. 
Hence we may apply the a priori estimate and Morrey's inequality as before to conclude that
\[
V \in C^k(I \times X \times M).
\]
\end{proof}

%%%%%%%%%%%%%%%%%%%%%%%%%%%%%%%%%%%%%%%%%%%%%%%%%%%%%%%%%%%%%%%

\section{Radial diffeomorphisms near boundary}

In this section we will construct a family of diffeomorphisms $G_x$ of
$M$, that push forward the standard measure $\mu$ to a family of
measures that, at least close to the boundary $\partial M$, coincide
with $\mu_x$.

Let us recall the assumptions made in Section
\ref{sec.not}.  There we noted that there exists a
reference density $f \in C^{k+2}(M)$ such that for every $x \in X$,
and every $m \in M$, $\rho(x,m) \geq f(m)$.  The goal is to find a
collar around the boundary of $M$, such that on this collar the
measures $\mu_x$ can be represented by monotone rearrangement along a
particular coordinate.

The collar can be parametrised through a map
\[
Q: \partial M \times [0,1] \to M.
\]
The volume form $f \text{d} m$ can be pulled back by this
parametrisation,
\[
\int_{Q(\partial M \times [0,1])} f \text{d} m = \int_{\partial M}
\int_0^1 f J Q \text{d} t \wedge \text{d} \Vol_{\partial M},
\]
where $J Q$ denotes the Jacobian of $Q$.

We will look for diffeomorphisms $G_x: \mathring{M} \to \mathring{M}$,
such that the push forwards $( G_x^{-1} )_{*}\mu_x$ coincide with $\mu$
at least close to the boundary of $M$, that is
\begin{equation}
\label{eq:GxUniformizesBd}
\left. (G_x^{-1})_{*} \mu_x \right|_{\partial M \times [0,1/3] } =
\left. \mu \right|_{\partial M \times [0,1/3]}.
\end{equation}
We require $G_x$ to be of a very specific form close to the boundary.
We will construct functions $\bar{g}_x: \partial M \times [0,1]
\to [0,1]$ such that
\[
G_x(a,t) = (a, \bar{g}_x(a,t)).
\]

For $a \in \partial M$, we define the related function $g_x( a, \cdot )$ by the integral equation
\begin{equation}
\label{eq:MasterEqForg}
\int_0^{g_x(a,t)} \rho(x,a,s) JQ(a,s) \text{d} s = \int_0^{t} f(a,s) JQ(a,s) \text{d} s.
\end{equation}
Note that if $\bar{g}_x (a,t)= g_x(a,t)$ for all $a \in \partial M$ and $0 \leq t \leq 1/3$, then indeed (\ref{eq:GxUniformizesBd}) holds.

We will use this integral equation to show that the maps $G_x$
obtained in this way will have the required regularity properties.  To
obtain the diffeomorphisms of $M$ we will need to modify the maps
$G_x$ in the interior of $M$.

\subsection{Regularity for the collar problem}

The next lemma shows uniform bounds on the regularity of $g_x(a,t)$ with respect to the boundary coordinate and the parameter $x$.
Note that under the current assumptions, uniform bounds on the derivatives in $t$ will typically not hold.

\renewcommand{\theenumi}{\roman{enumi}}
\begin{lemma}
\label{le:RegBoundary}
Suppose that there exist functions $E: \partial M \times [0,1] \to  (0,\infty)$, and $B:\partial M \times [0,1] \to  (0,\infty)$ such that for every multi-index $\beta$ and $j \in \mathbb{N}_0$ such that $|\beta| + j \leq k$, for every $x \in X$, every $a \in \partial M$ and $t \in [0,1)$,
\begin{equation}
\label{eq:RegBoundAssumption1}
\frac{1}{\rho(x,a,t) JQ(a,t)} \left| D_x^{\beta} D_t^{j} \left(\rho(x,a,t) JQ(a, t )\right)\right| \leq E(a,t)^{|\beta|} B(a,t)^{j}
\end{equation}
and
\begin{equation}
\label{eq:RegBoundAssumption2}
\frac{1}{\rho(x,a,t) JQ(a,t)} \int_0^t\left|D_x^{\beta} \rho(x,a,s)\right| JQ(a,s) \mathrm{d} s \leq \frac{E(a,t)^{|\beta|}}{B(a,t)}.
\end{equation}
Then for every multi-index $\beta$, such that $1 \leq |\beta| \leq k$, the function $g_x$ defined by (\ref{eq:MasterEqForg}) satisfies
\begin{equation}
\label{eq:RegBoundConcl}
\left| D_x^\beta g_x(a,t) \right| \leq C(|\beta|,n) \frac{E(a,g_x(a,t))^{|\beta|}}{B(a,g_x(a,t))}.
\end{equation}
In particular, if $E^j(t) \leq A B(t)$ for $j = 1, \dots, k$, and some constant $A > 0$, then
\[
\left| D_x^\beta g_x(a,t) \right| \leq C(|\beta|,n) A.
\]
\end{lemma}

\begin{proof}
First, we note that without loss of generality we may assume that the Jacobian $JQ$ equals $1$, by redefining $\rho$.
Moreover, we may suppress the dependence on the boundary coordinate $a \in \partial M$.

Consider then the defining equation for $g_x(t)$
\begin{equation}
\label{eq:masterRegBoundary}
\int_0^{g_x(t)} \rho(x,s) \mathrm{d} s = \int_0^t f(s) \mathrm{d} s.
\end{equation}
We will prove the Lemma by an inductive argument. 
For the base case, let $\beta$ be a multiindex with $|\beta|=1$, and take the derivative of both sides of equation (\ref{eq:masterRegBoundary}) with respect to $x$ to obtain
\[
D_x^{\beta} g_x(t) \rho(x, g_x(t)) = - \int_0^{g_x(t)} D_x^{\beta} \rho(x,s) \mathrm{d} s,
\]
for every multiindex $\beta$ with $|\beta|=1$.
The assumption (\ref{eq:RegBoundAssumption2}) applied to the right hand side of this equation immediately implies the conclusion for $|\beta| = 1$.

Next, we assume that the conclusion has been shown for all $\beta$ with $|\beta| \leq m - 1$ for some positive integer $m \leq k$. 
We then choose $\beta$ with $|\beta| = m$, and derive an equation for $D_x^\beta g_x(t)$ by taking $\beta$ derivatives with respect to $x$ in equation (\ref{eq:masterRegBoundary}). 
The resulting exact expression is lengthy, but for our purposes it suffices to realize that it is of the form
\begin{equation}
\label{eq:RegBoundGenFormDeriv}
\begin{split}
D_x^\beta g_x(t)  \rho(x,g_x(t)) &=
- \sum_{} a_{\beta_0,\beta_1, \dots, \beta_\ell} D_1^{\beta_0} D_2^{\ell} \rho(x, g_x(t)) D_x^{\beta_1} g_x(t) \cdots D_x^{\beta_{\ell+1}} g_x(t)\\
&\quad - \int_0^{g_x(t)} D_x^\beta \rho(x,s) \mathrm{d} s,
\end{split}
\end{equation}
where the coefficients $a_{\beta_0,\beta_1,\dots,\beta_\ell}$ are positive integers, and the sum is over indices that satisfy the following conditions
\[
\left\{
\begin{aligned}
|\beta_0| &\leq |\beta| - 1, \\
\ell &\leq |\beta| - |\beta_0| - 1, \\
1 &\leq |\beta_i| \leq |\beta| - 1, \qquad \text{ for all } i = 1,\dots, \ell + 1,\\
\beta&= \beta_0 + \cdots + \beta_{\ell+1}.
\end{aligned}
\right.
\]
Moreover, we have written $D_1^{\beta_0} D_2^\ell \rho (x , g_x(t))$ instead of $D_x^{\beta_0} D_t^\ell \rho(x, g_x(t))$ to emphasize that these are partial derivatives of the function $\rho$ evaluated in the given point.
The important features of the terms in the sum of the right hand side of the equation (\ref{eq:RegBoundGenFormDeriv}) are that there are precisely $|\beta|-|\beta_0| = \ell + 1$ factors of the form $D_x^{\beta_i} g_x(t)$. 
By the induction hypothesis,
\[
\left| D_x^{\beta_i} g_x(t) \right| \leq C(|\beta_i|, n ) \frac{E(g_x(t))^{|\beta_i|}}{B(g_x(t))},
\]
while by assumption
\[
\frac{1}{\rho(x,g_x(t))}\left| D_1^{\beta_0} D_2^\ell \rho(x, g_x(t)) \right|\leq E(g_x(t))^{|\beta_0|} B(g_x(t))^{\ell}.
\]
Since in particular $|\beta| = |\beta_0| + \cdots + |\beta_{\ell+1}|$, it follows that each of the terms in the sum of the right hand side of equation (\ref{eq:RegBoundGenFormDeriv}) can be estimated by
\[
\left| D_1^{\beta_0} D_2^{\ell} \rho(x, g_x(t)) D_x^{\beta_1} g_x(t) \cdots D_x^{\beta_{\ell+1}} g_x(t) \right|
\leq C(|\beta|,n) \frac{E(g_x(t))^{|\beta|}}{B(g_x(t))}.
\]
That the last term on the right hand side of equation (\ref{eq:RegBoundGenFormDeriv}) can be estimated by a term of this form follows from the assumption (\ref{eq:RegBoundAssumption2}).
Therefore, the conclusion (\ref{eq:RegBoundConcl}) also holds for $|\beta| = m$.
\end{proof}

\subsection{Construction of the diffeomorphisms rearranging the collar}

Based on the regularity properties derived in the previous section, we may now construct the diffeomorphisms $G_x$.

\begin{proposition}
\label{pr:DiffCollar}
Under the assumptions in Section \ref{suse:assumptions}, there exists a family of diffeomorphisms $G_x:\mathring{M} \to \mathring{M}$ with the following properties:
\begin{enumerate}
\item \label{item:DefOnCompl} $G_x(m) = m$ for all $m \in M \backslash (\partial M \times [0,2/3))$
\item \label{item:DefOnCollar} $G_x$ is radial, that is there exists $\bar{g}_x(a,t)$ such that $G_x(a,t) = (a, \bar{g}_x( a,t ) )$ for all $(a, t) \in \partial M \times [0,1)$.
\item \label{item:Agreegx} $\bar{g}_x(a,t) = g_x(a,t)$, for all $t \in [0,1/3]$, where $g_x$ is defined in (\ref{eq:MasterEqForg}), and therefore
\[
\left. (G_x^{-1})_{*} \mu_x \right|_{\partial M \times [0,1/3]} = \left. \mu \right|_{\partial M \times [0,1/3]}.
\]
\item \label{item:JointReg} The map $(x,m) \mapsto G_x(m)$ is of class $C^k(X \times \mathring{M}, \mathring{M})$.
\item \label{item:UnifRegx} The $C^k$-norm of the map $x \mapsto G_x(m)$ is bounded, uniformly in $m \in M$. 
\item \label{item:StayAwayFromBoundary} There exists a $t_*> 0$ such that $\bar{g}_x(a,1/6) \geq t_*$ for all $x \in X$. 
\item \label{item:DensityBound} The density of $(G_x^{-1})_{*} \mu_x$ with respect to $\mathrm{d} \Vol$ is uniformly bounded from below and bounded in $C^k(X\times M \backslash (\partial M \times[0,1/6]))$.
\item The diffeomorphism $G_x: \mathring{M} \to \mathring{M}$ extends to a homeomorphism $\overline{M}\to\overline{M}$, equal to the identity on $\partial M$.
\end{enumerate}
\end{proposition}

\begin{proof}
We will construct $G_x$ from the functions $g_x$ defined by (\ref{eq:MasterEqForg}). 
In order to ensure that $G_x$ is the identity on the complement of the collar, we need to interpolate $g_x$ to the identity.
Let $\eta:[0,1]\to [0,1]$ be a smooth, nonincreasing function, with $\eta(t) = 1 $ for $0\leq t \leq 1/3$ and $\eta(t) = 0$ for $t \geq 2/3$.
Define
\begin{equation}
\label{eq:Defgbar}
\bar{g}_x(a,t) = \eta(t) g_x(a,t) + (1- \eta(t)) t.
\end{equation}
For the following, we use that $f(a,t) < \rho(x,a,t)$ for all $(a,t) \in \partial M \times [0,1]$, and therefore $g_x(a,t) \leq t$. 
Indeed, this implies that
\[
\partial_t \bar{g}_x(a,t) = \eta'(t) (g_x(a,t) - t) + \eta(t) \partial_t g_x(a,t) + 1-\eta(t) > 0.
\]
We define $G_x(a,t) = (a, \bar{g}_x(a,t))$ for all $(a,t) \in \partial M \times [0,1)$, and $G_x(m) = m$ for $m \in M \backslash (\partial M \times [0,1))$, which immediately shows items (\ref{item:DefOnCompl}) - (\ref{item:Agreegx}).

Item (\ref{item:UnifRegx}), the uniform boundedness of the $C^k$-norm of the map $x \mapsto G_x(m)$, was established in Lemma \ref{le:RegBoundary}. 
It is easier to prove the regularity of the map $(x,m) \mapsto G_x(m)$ as stated in (\ref{item:JointReg}). 
The required regularity of $g_x$ immediately follows from taking $j$ derivatives in the defining equation for $g_x$
\[
\int_0^{g_x(a,t)} \rho(x,a,s) J Q(a,s) \mathrm{d} s = \int_0^t f(a,s) JQ(a,s) \mathrm{d} s.
\]
Unlike for the maps $x \mapsto G_x(m)$, we now do not require uniform bounds on the derivatives as $m$ approaches the boundary.
The regularity of $\bar{g}_x(a,t)$ and therefore of $G_x(t)$ follows from its definition (\ref{eq:Defgbar}) and the regularity of the cutoff function $\eta$.

Item (\ref{item:StayAwayFromBoundary}) is a direct consequence of the continuity of the maps $(x,m) \mapsto G_x(m)$ and the compactness of $X$ and $\partial M \times[1/3,1]$.

Since $G_x$ is the identity on $M \backslash (\partial M \times [0,2/3))$, by the compactness of $X$, the densities of $(G_x^{-1})_{*} \mu_x$ are certainly bounded on $M \backslash (\partial M \times [0,2/3))$.
If $\nu(x,a,t)$ denotes the density of $(G_x^{-1})_{*} \mu_x$ with respect to $\mathrm{d} \Vol$, the following equation holds for $\nu$,
\[
\nu(x,a,t) JQ (a,t) = \rho(x,a,\bar{g}_x(a,t)) JQ(a, \bar{g}_x(a,t) ) \partial_t \bar{g}_x(a,t).
\]
Note that $\bar{g}_x(a,t) \geq g_x(a,t) > c$ for every $(a,t) \in \partial M \times [1/6,1]$. 
Therefore both $\rho(x,a,g_x(a,t))$ and $\rho(x,a,\bar{g}_x(a,t))$ are uniformly bounded away from zero. 
On the other hand,
\[
f(x,a,t) JQ(a,t) = \rho(x,a,g_x(a,t)) JQ(a, g_x(a,t)) \partial_t g_x(a,t),
\]
which implies that $\partial_t g_x$ is bounded away from zero as well, which in turn implies that $\partial_t \bar{g}_x(a,t)$ is uniformly bounded away from zero. 
This finishes the proof of (\ref{item:DensityBound}).

Finally, from the definition of $g_x$ it is also clear that $g_x(a,t)\to a$ as $t \downarrow 0$. 
Therefore, $G_x$ extends to a homeomorphism from $\bar{M}$ to itself.

\end{proof}

\section{Proof of the Main Theorem}

We will now combine the results of the previous two sections to give a proof of Theorem \ref{th:CkRepresentability}.

\begin{proof}[Proof of Theorem \ref{th:CkRepresentability}]
  As mentioned above, we will construct transport maps $T_x: M \to M$
  as a composition of a diffeomorphism $G_x: M \to M$ and maps $F_x$.
  The diffeomorphism $G_x$ is provided by Proposition
  \ref{pr:DiffCollar}.  Define $\tilde{\mu}_x = (G_x^{-1})_{*} \mu_x$.
  The behavior of $\tilde{\mu}_x$ in the collar $\partial M \times
  [0,1/3]$ is independent of $x$ and coincides with $f \text{d} \Vol$.
  We select a neighbourhood $V$ of $\partial M$, contained in $\partial
  M \times [0,1/3]$ but such that $\partial M \times[0,1/6] \subset
  V$, and such that the boundary of $M \backslash V$ is of class
  $C^{k+2}$.  By the conclusions of Proposition \ref{pr:DiffCollar},
  we may apply Proposition \ref{p:moser1} to the family of measures
  $\tilde{\mu}_x$, restricted to the complement of the collar $M
  \backslash V$.  This way we obtain a family of diffeomorphisms $F_x:
  M \backslash V \to M \backslash V$, for which the derivatives are
  uniformly bounded, and such that
\[
(F_x)_{*} \mu = \tilde{\mu}_x = (G_x^{-1})_{*} \mu.
\]
We extend the maps $F_x$ thus obtained by the identity on $V$.
It follows that 
\[
(T_x)_{*} \mu = (G_x)_{*}(F_x)_{*} \mu = (G_x)_{*} (G_x^{-1})_{*} \mu_x = \mu_x.
\] 

For $m \in \partial M \times [0,1/3]$, the image $F_x(m) = m$ is in
particular independent of $x$.  The uniform bounds on the derivatives
of the map $x \mapsto G_x(m)$, shown in Proposition
\ref{pr:DiffCollar} then imply that the derivatives of $G_x(F_x(m))$
are uniformly bounded.  For $m \in M \backslash V$, we may use that the derivatives of $(x,m) \mapsto G_x(m)$ are
bounded by compactness of $X \times (M \backslash V)$.
\end{proof}

\section{Examples}
\label{se:Examples}

We would now like to give examples of families of probability measures that are not boundedly $C^1$-representabe.
In both examples, the families have regular densities. In the first example, the order of decay towards the boundary is different for different values of the parameter.

\begin{example}
\label{ex:NonUniformDecay}
Let $M=[0,1]$ and $X=[-1,1]$ and 
\[
\rho(x,m) = 2 x^2 m + 5(1-x^2) m^4. 
\]
Suppose that in fact, a representation $T:X \times \Omega \to M$ exists, such that 
\[
T(x, \cdot)_{*} \mathbb{P} = \rho(x ,\cdot) \mathrm{d} m,
\]
for a probability measure $\mathbb{P}$ on $\Omega$. 
Given $x$, 
\[
\mathbb{P} \{ \omega \in \Omega \, | \, T(x, \omega) \leq |x| / 2 \} \geq 4 x^4,
\]
while 
\[
\mathbb{P} \{ \omega \in \Omega \, | \, T(0, \omega) \leq |x|^{4/5} \} = x^4.
\]
As a consequence, for any $x \neq 0$ small enough, there exists a set of $\omega \in \Omega$ of positive measure, such that
\[
\frac{| T(x,\omega) - T(0,\omega) |}{|x|} \geq \frac{1}{2 |x|^{1/5}},
\]
showing that the first derivatives of $|T(x,\omega)|$ are not bounded uniformly in $\omega$. 
In this example, assumption (\ref{eq:DerivativeAssumption}) fails for $\beta = 1$ and $j=0$. 
\end{example}

In the second example, the dependence on $x$ is very mild and in particular, the qualitative behavior of the decay does not change with $x$. However, the decay rate of the density towards the boundary is not well defined. 

\begin{example}
\label{ex:DecayNotWellDefined}
For a second example in which bounded $C^1$ representability does not hold, again consider $M = [0,1]$ and $X = [-1,1]$,
but now define
\[
\rho(x,m) := (2 + x) m^5 \sin^2(1/m) + m^{30} + \tilde{\rho}(x,m),
\]
where $\tilde{\rho}(x,m)$ is there for normalization purposes, and is assumed to be smooth on $X \times M$ and zero for $m \leq 1/2$. 

Note also that $\rho \in C^2(X \times M)$. However, a bounded $C^1$ representation does not exist. Suppose on the contrary that such a representation $T: X\times M \to M$ exists.
We give the following heuristic argument, that will be made rigorous at the end of the section. The flux through $m_k:=1/(\pi k)$ is of the order $m_k^6$, while the density is merely $m_k^{30}$. 
It follows that the average velocity of the particles passing through $m_k$ is of the order $m_k^{-24}$, which can be made arbitrarily large by picking $m_k$ small enough.

In this example, no functions $E$ and $B$ can be found that simultaneously satisfy the assumptions (\ref{eq:DerivativeAssumption}-\ref{eq:ClosureAssumption}), due to the strong blow up of the left-hand-side in assumption (\ref{eq:IntegratedAssumption}).
\end{example}

\subsection{Uniform decay towards the boundary}

To get further intuition for the assumptions, let us consider the special case of $M = [0,1]$, and $X = [0,1]$.
We write $\rho(x,t) = \nu(x,t) e^{-H(t)}$ for a convex, decreasing function $H:[0,1]\to \mathbb{R}$, such that $\lim_{t\downarrow 0} H(t) = \infty$.
Suppose $\nu \in C^k(X\times M)$, and $\nu \geq 1$. 

Let us first try to estimate the expression appearing in assumption (\ref{eq:IntegratedAssumption}),
\[
\begin{split}
\frac{1}{\rho(x,t)} \int_0^t |D_x^\beta \rho(x,s)| \text{d} s
& \leq \frac{1}{\nu(x,t)} \int_0^t |D_x^\beta \nu_x| e^{-(H(s)-H(t)) } \text{d} s \\
\end{split}
\]
Since $H$ is convex,
\[
H'(t)(s-t) \leq H(s) - H(t).
\]
Therefore
\[
\begin{split}
\frac{1}{\rho(x,t)} \int_0^t |D_x^\beta \rho(x,s)| \text{d} s 
&\leq C \int_0^t e^{-H'(t) (s-t)} \text{d} s \\
&\leq - C \frac{1}{H'(t)}.
\end{split}
\]
Therefore, in order to satisfy assumption (\ref{eq:IntegratedAssumption}), one could choose for instance $E \equiv C$, and $B(t) = - 1/H'(t)$. 
With these choices, assumption (\ref{eq:DerivativeAssumption}) would require that
\begin{equation}
\label{eq:BoundInTermsH}
|H^{(k)} (t) |\leq C (H'(t))^k.
\end{equation}
This condition is easily verified for a variety of functions, such as $e^{-H} = t^\alpha$, and for $H = t^{-\alpha}$ for some $\alpha > 0$. However, there are other examples of $H$ for which this condition fails, but with different choices of $E$ and $B$ the assumptions can still be made to hold.

Indeed, if $H(t) = - \alpha \ln(t)$ for some $\alpha > 0$, that is, if the density $e^{-H}$ has a power decay towards the boundary of $M$, 
\[
\left(\frac{1}{(H'(t))^j}\right)^{(j)}  = \left((-1)^j \frac{t^j}{\alpha^j} \right)^{(j)} = (-1)^j \frac{j!}{\alpha^j},
\]
so that the assumptions on $H$ hold with the function $E$ identical to $1$. 

Similarly, when $H(t) = t^{-\alpha}$, for $\alpha > 0$, 
\[
\left( \frac{1}{(H'(t))^j} \right)^{(j)} = \left( (-1)^j \alpha^j t^{j(\alpha+1)} \right)^{(j)} = (-1)^j \alpha^j  t^{j \alpha}
\]
the function $E$ identical to $1$ works as well.

However, there are examples of $H$ for which (\ref{eq:BoundInTermsH}) does not hold, for instance for
\[
H(t) = \log( - \log(t) ).
\]
In that case, the assumptions allow for enough flexibility to still be able to show representability. 
Indeed, one could choose $B(t) = C_1/t$, and $E(t) = C_2 \log(t)$.

\section{Further observations}

In this section we give two necessary conditions for bounded representability. 
Although both follow from very simple observations, they can be useful to show that certain families of probability measures cannot be boundedly represented.

\subsection{Obstruction to bounded Lipschitz representability}

Both Example \ref{ex:NonUniformDecay} and \ref{ex:DecayNotWellDefined} exhibit a similar obstruction to bounded representability, which we extract as Theorem \ref{th:obstructionRepr} below. 
Before stating the theorem, we introduce some notation. 

Let $\Pi( \mu, \bar{\mu})$ denote the family of Borel probability measures on $M \times M$ with marginals $\mu$ and $\bar{\mu}$. 
We consider the supremum Wasserstein distance on $M$ given by
\[
W_\infty( \mu, \bar{\mu} ) = \inf_{ \pi \in \Pi( \mu, \bar{\mu} ) } \underset{(m_1, m_2) \in (M \times M, \pi)}{\esssup}  d_M( m_1, m_2). 
\]

\begin{theorem}
\label{th:obstructionRepr}
Let $\{\mu_x\}_{ x \in X }$ be a family of probability measures on a complete, separable metric space $M$, parametrized by a metric space $X$.
If the family $\{ \mu_x \}_{x\in X}$ is boundedly Lipschitz representable then
\[
\sup_{x,y \in X} \frac{W_\infty( \mu_x, \mu_y )}{d_X(x,y)} < \infty.
\]
\end{theorem}

\begin{proof}
Suppose that a constant $L >0$ and a representation $T : X \times \Omega \to X$ exist, such that the Lipschitz constant of the map $T(\cdot, \omega)$ is bounded by $L$ for $\mathbb{P}$-a.e. $\omega \in \Omega$. 
Let $x, y \in X$ and note that the representation $T$ induces a transportation plan
\[
(T(x, \cdot) ,T(y, \cdot))_{*} \mathbb{P} \in \Pi(\mu_x, \mu_y).
\]
Since $T$ is a bounded Lipschitz representation, 
\[
\mathbb{P} \left\{ \omega \in \Omega \, | \, \frac{d_M( T(x, \omega), T(y, \omega) )}{d_X(x,y)} > L \right\} > 0,
\]
from which it follows that
\[
\frac{W_\infty( \mu_x, \mu_y )}{d_X(x,y)}\leq L.
\]
\end{proof}

If $M =[0,1]$, the distance $W_\infty(\mu, \bar{\mu})$ is easily calculated with use of the distribution functions 
\[
F_\mu(x) := \mu( [0,x] ).
\]
If we define the generalized inverse of a monotone function $F$ by
\[
F^{-1}(t) = \inf\left \{m \in [0,1] \,| \, F(x) > t \right\},
\]
then 
\[
W_\infty(\mu, \bar{\mu}) = \sup_{t \in [0,1]} \left| F_\mu^{-1}(t) - F_{\bar{\mu}}^{-1}(t) \right|,
\]
and this optimal value is reached by monotone rearrangement (although in general it is assumed by many other transportation plans as well).

Let us recheck that bounded $C^1$ representation is impossible in Example \ref{ex:DecayNotWellDefined}. 
Let again $m_k := 1 / (k \pi)$ for some large $k \in \mathbb{N}$. 
Then we may choose $\delta > 0$ such that (writing $F_x$ for $F_{\mu_x}$), 
\[
F_{x}(m_k + \delta) - F_x( m_k) < 2 \delta m_k^{30}.
\]
On the other hand, for some constant $C > 0$,
\[
F_{x+ C \delta m_k^{24}}(m_k) - F_x(m_k) \geq 4 \delta m_k^{30},
\]
so that
\[
F_{x+C \delta m_k^{24}}^{-1}(F_x(m_k)) - F_x^{-1}(F_x(m_k)) \geq \delta,
\]
and therefore
\[
\frac{W_\infty( \mu_x, \mu_{x + \delta m_k}) }{C \delta m_k^{24} } \geq \frac{1}{C m_k^{24}},
\]
which becomes arbitrarily large as $m_k \downarrow 0$.

\subsection{Obstruction to bounded $C^k$-representability}

We conclude with a necessary condition for bounded $C^k$-representability. 
In some situations, this condition can be used to quickly rule out the existence of a bounded $C^k$ representation.

\begin{proposition}
\label{pr:ObstructionTesting}
Let $M$ and $X$ be compact Riemannian manifolds of class $C^{k}$ and let $\{\mu_x\}_{x \in X}$ be a family of Borel probability measures on $M$. If the family $\{ \mu_x\}_{x \in X}$ is boundedly $C^k$ representable, then for every function $h \in C^k(M)$, the function $E_h: X \to \mathbb{R}$ defined by
\[
E_h(x) := \int_M h \mathrm{d} \mu_x
\]
is also of class $C^k(M)$.
\end{proposition}

\begin{proof}
Let $(\Omega, \mathbb{P})$ be a probability space and suppose a representation $T:X \times \Omega \to M$ exists such that for $\mathbb{P}$-a.e. $\omega \in \Omega$, 
the $C^k$-norm of the map $T(\cdot, \omega)$ is less than $C$. 
Since $T(x, \cdot)_* \mathbb{P} = \mu_x$, the following identity holds
\[
\begin{split}
E_h(x) &= \int_\Omega h(T(x,\omega) ) \mathrm{d} \mathbb{P}(\omega).
\end{split}
\]
Note the $C^k$-norm of the function $h(T(x,\omega))$ is uniformly bounded. It follows that $E_h$ is of class $C^k(M)$ by the dominated convergence theorem.
\end{proof}

%%%%%%%%%%%%%%%%%%%%%%%%%%%%%%%%%%%%%%%%%%%%%
\section*{Acknowledgements}
The research leading to these results
has received funding from the European Research Council under the
European Union's Seventh Framework Programme (FP7/2007-2013) / ERC
grant agreement n$^\circ$~267087. C.S.R. has also been supported by the Brazilian agency: grant \#2015/02230-9 and grant \#2016/00332-1, S\~{a}o Paulo Research Foundation (FAPESP).

%%%%%%%%%%%%%%%%%%%%%%%%%%%%%%%%%%%%%%%%%%%%%

\bibliographystyle{amsalpha}

\end{document}